\def\ams{1}

\if\ams1

\documentclass{amsart}
\usepackage[numbers,sort]{natbib}
\usepackage[pdftitle={Bayesian inference in Epidemics},
pdfauthor={Samuel Bronstein, Stefan Engblom, Robin Marin},
pdffitwindow=true,
breaklinks=true,
colorlinks=true,
linkcolor=red,
anchorcolor=red,
citecolor=blue,urlcolor=blue,filecolor=blue,backref=page]{hyperref}
\usepackage[foot]{amsaddr}

\else

\documentclass{aims}
\usepackage[numbers,sort]{natbib}
\usepackage{txfonts}

\fi

\usepackage{amssymb}
\usepackage{amsmath}
\usepackage{amsthm}
\usepackage{bbm}
\usepackage[latin1]{inputenc}
\usepackage[english]{babel}
\usepackage{graphicx}

\usepackage[rightcaption]{sidecap}
\sidecaptionvpos{figure}{c}

\usepackage{numprint} 
\npthousandsep{,}

\numberwithin{equation}{section}
\numberwithin{table}{section}
\numberwithin{figure}{section}

\if\ams1

\theoremstyle{plain}
\newtheorem{theorem}{Theorem}[section]
\newtheorem{lemma}[theorem]{Lemma}
\newtheorem{proposition}[theorem]{Proposition}
\newtheorem{corollary}[theorem]{Corollary}
\newtheorem{conjecture}[theorem]{Conjecture}

\theoremstyle{definition}
\newtheorem{definition}{Definition}[section]
\newtheorem{assumption}[definition]{Assumption}
\newtheorem{convention}[definition]{Convention}
\newtheorem{example}[definition]{Example}

\theoremstyle{remark}
\newtheorem*{remark}{Remark}
\newtheorem*{note}{Note}
\newtheorem*{claim}{Claim}

\newcommand{\Ddata}{D}
\newcommand{\Ndata}{N}

\newcommand{\Realdom}{\mathbf{R}}
\newcommand{\one}{\mathbf{1}}
\newcommand{\normal}{\mathcal{N}}

\newcommand{\OU}{$\mathcal{OU}$}

\DeclareMathOperator{\Expect}{\mathbb{E}}
\DeclareMathOperator{\Cov}{\mathbb{C}ov}

\DeclareMathOperator{\KL}{D_{KL}}
\DeclareMathOperator{\diag}{\text{Diag}}

\renewcommand{\Pr}{\mathbf{P}}

\newcommand{\updated}[1]{#1}

\else

\theoremstyle{plain}
\newtheorem{theorem}{Theorem}[section]
\newtheorem{lemma}[theorem]{Lemma}
\newtheorem{proposition}[theorem]{Proposition}

\theoremstyle{definition}
\newtheorem{definition}{Definition}[section]
\newtheorem{assumption}[definition]{Assumption}

\theoremstyle{remark}
\newtheorem*{remark}{Remark}

\newcommand{\Ddata}{D}
\newcommand{\Ndata}{N}

\newcommand{\Realdom}{\mathbf{R}}
\newcommand{\one}{\mathbf{1}}
\newcommand{\normal}{\mathcal{N}}

\newcommand{\OU}{$\mathcal{OU}$}

\DeclareMathOperator{\Expect}{\mathbb{E}}
\DeclareMathOperator{\Cov}{\mathbb{C}ov}

\DeclareMathOperator{\KL}{D_{KL}}
\DeclareMathOperator{\diag}{\text{Diag}}

\renewcommand{\Pr}{\mathbf{P}}

\newcommand{\updated}[1]{\textcolor{red}{#1}}

\fi

\begin{document}

\if\ams1
\title[Bayesian inference in Epidemics]{Bayesian inference in
  Epidemics: linear noise analysis}

\author[S. Bronstein]{Samuel Bronstein$^1$}
\address{$^1$Department of Mathematics and Applications, ENS Paris,
  75005 Paris, France.}

\author[S. Engblom]{Stefan Engblom$^2$}
\thanks{Corresponding author: S. Engblom, telephone +46-18-471 27 54,
  fax +46-18-51 19 25, URL: \url{http://user.it.uu.se/~stefane}.}

\author[R. Marin]{Robin Marin$^2$}
\address{$^2$Division of Scientific Computing \\
  Department of Information Technology \\
  Uppsala University \\
  SE-751 05 Uppsala, Sweden.}

\email{samuel.bronstein@ens.fr, stefane@it.uu.se, robin.marin@it.uu.se}

\date{\today}

%
%
%


\begin{abstract}
  This paper offers a qualitative insight into the convergence of
  Bayesian parameter inference in a setup which mimics the modeling of
  the spread of a disease with associated disease
  measurements. Specifically, we are interested in the Bayesian
  model's convergence with increasing amounts of data under
  measurement limitations. Depending on how weakly informative the
  disease measurements are, we offer a kind of `best case' as well as
  a `worst case' analysis where, in the former case, we assume that
  the prevalence is directly accessible, while in the latter that only
  a binary signal corresponding to a prevalence detection threshold is
  available. Both cases are studied under an assumed so-called linear
  noise approximation as to the true dynamics. Numerical experiments
  test the sharpness of our results when confronted with more
  realistic situations for which analytical results are unavailable.
\end{abstract}

\maketitle

\else

\title{Bayesian inference in epidemics: linear noise analysis}

\author{%
  Samuel Bronstein\affil{1},
  Stefan Engblom\affil{2,}\corrauth
  and
  Robin Marin\affil{2}
}

\shortauthors{the Author(s)}

\address{%
  \addr{\affilnum{1}}{Department of Mathematics and Applications, ENS Paris,
    75005 Paris, France.}
  \addr{\affilnum{2}}{Division of Scientific Computing, Department of
    Information Technology, Uppsala University, SE-751 05 Uppsala, Sweden.}}

\corraddr{stefane@it.uu.se; Tel: +46-18-471 27 54; Fax:\\ +46-18-51 19 25.}

\begin{abstract}
  This paper offers a qualitative insight into the convergence of
  Bayesian parameter inference in a setup which mimics the modeling of
  the spread of a disease with associated disease
  measurements. Specifically, we are interested in the Bayesian
  model's convergence with increasing amounts of data under
  measurement limitations. Depending on how weakly informative the
  disease measurements are, we offer a kind of `best case' as well as
  a `worst case' analysis where, in the former case, we assume that
  the prevalence is directly accessible, while in the latter that only
  a binary signal corresponding to a prevalence detection threshold is
  available. Both cases are studied under an assumed so-called linear
  noise approximation as to the true dynamics. Numerical experiments
  test the sharpness of our results when confronted with more
  realistic situations for which analytical results are unavailable.
\end{abstract}


\maketitle

\fi


\section{Introduction}

Computational models in epidemics are commonly relied upon to estimate
the disease spread at fairly large spatial- and temporal scales, often
referred to as \emph{scenario generation}. With the increasing volumes
and improved resolution of data from, e.g., mobile apps and disease
testing time series from hospitals and nursing homes, predictive
data-driven models formed from first principles are within reach. The
accuracy of such models is ultimately limited by the specifics of the
available disease surveillance data. In this paper we attempt to gain
a qualitative understanding of how Bayesian inference of
epidemiological parameters may be expected to perform and what the
limiting factors are.

%

Epidemiological models are typically formed by postulating laws for
the flow of individuals between different \emph{compartments} in a
large population and have been studied in this form for a long
time. When connected with data in the form of observations, the
associated inference problem is also a fairly mature field, see,
e.g., \cite{keeling2011modeling, mckinley2009inference,
  andersson2012stochastic}. With the increasing qualities and
quantities of data, the various incarnations of data-driven modeling
allow for substantially higher modeling resolution compared to
traditional macroscopic approaches \cite{eubank2004modelling,
  ferguson2005strategies, balcan2009multiscale,
  merler2011determinants, Brooks-Pollock2014ABCSMC}. For example,
individual-level contact tracing has been used to study disease spread
models at various population sizes \cite{stehle2011simulation,
  bajardi2012optimizing, salathe2010high, obadia2015detailed,
  toth2015role}. Data-driven models have allowed epidemic and endemic
conditions to be investigated at a level of detail not previously
possible \cite{zhang2017spread, liu2018measurability, siminf3}.

The identification of the epidemiological parameters from data falls
under the scope of problems formally studied in \emph{System
  Identification}~\cite{soderstrom1989system}. However, a `system'
viewpoint of epidemiological modeling is not yet standard, and
identification of parameters is rather more often approached through
calibration of residuals~\cite{ferguson2005strategies, siminf3,
  fournie2018dynamic}, sometimes also blending in aspects of Bayesian
arguments. Fully Bayesian approaches are rarer, albeit with some
exceptions~\cite{Brooks-Pollock2014ABCSMC, engblom2020bayesian},
typically due to the technical difficulties with formulating suitable
(pseudo-)likelihoods and the slow convergence associated with the
conditioning of the problem. Although Bayesian inference is notably
well-posed thanks to its use of prior distributions, the posterior
distribution itself is often a computationally ill-conditioned object
whenever strong parameter correlations are present, e.g., resulting
from nearly singular maps from parameters to observables. These
conditions, together with the societal importance of this modeling
domain, make it relevant to reason around the limits of fully Bayesian
techniques.

Fundamental questions concerning Bayesian convergence in general
settings, including infinite-dimensional ones, have been treated
\cite{shen2001rates, stuart2010inverse}, and have also been revisited
with specific tools and applications in mind, e.g., Gaussian processes
and other machine learning algorithms \cite{stuart2018posterior,
  latz2020well}. ``Brittleness'' \cite{owhadi2015brittleness,
  sprungk2020local}, or high sensitivity to small perturbations have
been proposed to be a problematic phenomenon under certain
conditions. For applications, this points to the importance of
striking a balance between the granularity of the model and the
information content and level of detail of the available data.

With the specific aim of reaching qualitative conclusions for Bayesian
modeling in epidemiology, we will connect some usual modeling
approaches in infectious disease spread with a basic linear stochastic
differential equation. We analyze the pre-asymptotic Bayesian
posterior in the high-quality data regime as well as the estimate's
convergence under weakly informative disease measurements. This is
motivated by current trends in disease spread monitoring through for
example sewage water analysis and symptoms data collection using
smartphones \cite{wwGreece, CSSSWhite}.

The rest of the paper is organized as follows. In \S\ref{sec:SDEmeta}
we suggest the Ornstein-Uhlenbeck process as a meta-model of
epidemiological models, covering various
\updated{Susceptible-Infected-Recovered (SIR)-type models} locally in
time, during endemic as well as under epidemic conditions. In
\S\ref{sec:UnfilteredState} we briefly analyze the posterior
convergence when accurate full state measurements are available. The
main results are found in \S\ref{sec:FilteredState} where we more
fully develop an analysis for the case of poorly informative data. We
offer some examples of relevance in \S\ref{sec:Illustrations} and a
summarizing discussion is found in \S\ref{sec:conclusions}.


\section{The Ornstein-Uhlenbeck process}
\label{sec:SDEmeta}

\updated{We connect in this section the Ornstein-Uhlenbeck process
  with some basic epidemiological models via linearization and quasi
  steady-state arguments. We also briefly discuss, by means of a
  backward analysis, how continuous and discrete time in this setting
  can be connected in probability law under a certain parameter map.}

\subsection{\updated{Disease spread modeling}}

Epidemiological modeling typically involves ordinary differential
equations (ODEs), e.g., the SIR-model which is often used to model the
annual flu \cite{kermack1927contribution, keeling2011modeling},
\begin{align}
  \label{eq:SIR}
  \left. \begin{array}{rl}
           S'(t) &= -(\beta \Sigma^{-1}) S(t) I(t) \\
           I'(t) &= (\beta \Sigma^{-1}) S(t) I(t) - \gamma I(t) \\
           R'(t) &= \gamma I(t)
  \end{array} \right\}
\end{align}
in terms of Susceptible, Infected, and Recovered individuals,
respectively, and for a total population size $\Sigma \equiv
S+I+R$. The model parameters $\beta$ and $\gamma$ define the
transmission and recovery rates, and the basic reproduction number is
then given by $R_0 = \beta/\gamma$, that is, the expected number of
new infections resulting from a single index case
\cite{keeling2011modeling}.

Simplifying the SIR-model by removing the $R$-compartment one obtains
the SIS-model, \updated{where the recovered state has been removed and
  hence effectively identified with the susceptible state,}
\begin{align}
  \label{eq:SIS}
  \left. \begin{array}{rl}
           S'(t) &= -(\beta \Sigma^{-1})  S(t) I(t) + \gamma I(t) \\
           I'(t) &= (\beta \Sigma^{-1}) S(t) I(t) - \gamma I(t) \\
  \end{array} \right\}
\end{align}
and with the same $R_0$ as the SIR-model. It is sometimes useful to
add an environmental compartment expressing the \emph{infectious
  pressure} $\varphi$, i.e., the amount of infectious substance per
unit of space. This defines the SIS$_{\text{E}}$-model
\cite{siminf_Ch}, \updated{where the $E$ stands for the Environmental
  compartment,}
\begin{align}
  \label{eq:SIS_E}
  \left. \begin{array}{rl}
           S'(t) &= -\beta  S(t) \varphi(t) + \gamma I(t) \\
           I'(t) &= \beta S(t) \varphi(t) - \gamma I(t) \\
           \varphi'(t) &= \Sigma^{-1} I(t) -\rho \varphi(t)
  \end{array} \right\}
\end{align}
where the indicated governing equation for $\varphi$ is just a basic
example. The SIS$_{\text{E}}$-model is convenient to adapt to spread
over a network and to detection situations involving sampling the
environment \cite{engblom2020bayesian}. Due to the indirect
transmission the basic reproduction number now scales as a square
root, with $R_0 = \sqrt{\beta/(\gamma\rho)}$.

The use of ODEs can be justified for epidemiological models in
sufficiently large populations. In smaller populations, e.g., networks
of small communities, stochastic variants are necessary to properly
capture the underlying dynamics of the spread
\cite{britton2010stochastic, andersson2012stochastic}. For example, a
\updated{stochastic differential equation (SDE)-version} of
\eqref{eq:SIS} was analyzed in \cite{gray2011stochastic} which
essentially replaces $\beta \, dt$ by a Brownian diffused version
$\beta \, dt+\eta \, dB(t)$. This has the effect of lowering the basic
reproduction number to $R_0 = (\beta-\eta^2/2)/\gamma$. A
first-principle stochastic approach is to rather express the dynamics
as a continuous-time discrete state Markov chain (CTMC). The SIS-model
above is then defined via discrete transitions with exponentially
distributed waiting times,
\begin{align}
  \label{eq:ctmc}
  \left. \begin{array}{rcl}
           S + I &\xrightarrow{\beta\Sigma^{-1}} & 2 I \\
           I     &\xrightarrow{\gamma}& S
         \end{array} \right\}
\end{align}
meaning, e.g., that one infectious individual may infect one
susceptible individual, and such that $\beta$ and $\gamma$ are now
understood as rate parameters in the driving Poissonian processes. One
can show that now $R_0 = (1-\Sigma^{-1}/2) \times \beta/\gamma$ such
that consistency with the previous SDE formulation would require that
$\eta^2 = \Sigma^{-1} \beta$ and hence follows from the typical
Poissonian population-dependent noise scaling
$\eta \propto \Sigma^{-1/2}$.

Let $P(t)$ be some given measure of the intensity of the disease, such
as the absolute or relative disease prevalences, $I(t)$ or
$\Sigma^{-1}I(t)$, respectively. Another alternative could be the
infectious pressure $\varphi(t)$, which also informs on the current
disease intensity. Gathering data from the disease spread means
collecting information about $P(t)$, typically in the form of a
time-series, $(F(P(t_i))_i$, say, for some measurement operator
$F$. Consider \emph{endemic} conditions first, that is, $P$ is
considered stationary. As an ansatz, suppose
\begin{align}
  \label{eq:rho}
  P(t_i) &\sim \Pr[P = p] = \rho_\infty(p)
      \propto \exp\left( -2/\sigma^2 \, E(p) \right),
\end{align}
for some \emph{epidemiological potential} $E$. Under endemic
conditions we can expect this to be a single well potential, say,
\begin{align}
  \label{eq:sqwell}
  E(p) &\propto (p-\alpha)^2+\mbox{constant,}
\end{align}
or at least can locally be approximated by this form. This is
consistent with the \emph{Ornstein-Uhlenbeck (\OU-) process}
\cite{uhlenbeck1930theory},
\begin{align}
  \label{eq:OU1}
  X(0) &= X_0, \\
  \label{eq:OU2}
  dX(t) &= k-\mu X(t) \, dt+\sigma \, dW(t),
\end{align}
where $X(t) \in \Realdom$ and where the parameters of the model are
$\theta = [k,\mu,\sigma] \in \Realdom_+^3$ and in \eqref{eq:sqwell},
$\alpha = k/\mu$.

As a concrete example, we may approximate the continuous-time Markov
chain \eqref{eq:ctmc} by linearizing the rates around the non-trivial
stationary state $I_\infty = \Sigma (1-\gamma/\beta)$ of
\eqref{eq:SIS}. Similarly, the noise term $\sigma$ can be determined
by inspecting the total variance of the Poissonian rates in
\eqref{eq:ctmc} around this equilibrium. This yields a \emph{linear
  noise} \OU-approximation to the state variable $I$ of the SIS-model
\eqref{eq:ctmc} with parameters
\begin{align}
  \label{eq:SIS2OU}
  &\left. \begin{array}{rll}
      k &= \Sigma (\beta-\gamma)^2/\beta &= \gamma \Sigma
                                           (R_0-1)^2/R_0 \\
      \mu &= \beta-\gamma &= \gamma(R_0-1) \\
      \sigma^2 &= 2\Sigma (\beta-\gamma)^2/\beta &= 2\gamma\Sigma (R_0-1)/R_0
          \end{array} \right\} \\
  \intertext{where we use the approximation $R_0 = (1-\Sigma^{-1}/2)
  \times \beta/\gamma \approx \beta/\gamma$. Hence, \emph{with this specific
  interpretation of the \OU-process,}}
  \label{eq:SIS2OU_R0}
  &R_0 = \left(1-\Sigma^{-1}k/\mu\right)^{-1} = \left(1-\Sigma^{-1}\alpha\right)^{-1}.
\end{align}
The quality of this approximation is exemplified in
Fig.~\ref{fig:SIS_illu}; the effect of linearizing around the
stationary state can be seen as a slightly too fast transient compared
to the Markov chain.

\begin{figure}[t]
  \includegraphics{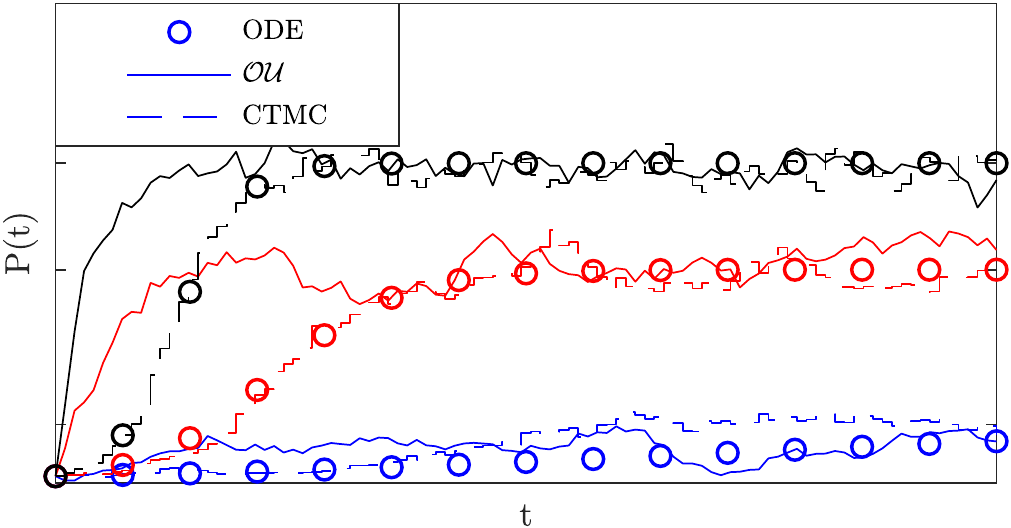}
  \caption{The approximation \eqref{eq:SIS2OU} exemplified for
    $R_0 = [1.1,1.5,2]$ (bottom and up), with $\Sigma = 1000$ and
    $\gamma = 1$. For comparison the ODE-, the \OU-, and the
    continuous-time Markov chain (CTMC) interpretation are shown.}
  \label{fig:SIS_illu}
\end{figure}

During \emph{epidemic} conditions we are rather sampling a transient
of the process and, moreover, there are also typically many kinds of
feedback involved, e.g., from minor adjustments of individual level
behavior, to major societal changes and governmental intervention
strategies. Although this means that there is now a greater challenge
in formulating a data-centric meta-model of the situation we may still
consider a window of time for which the disease spread parameters are
approximately constant and such that \eqref{eq:OU1}--\eqref{eq:OU2}
remain a relevant model of the situation. A difference is then that
data is gathered out of equilibrium (hence away from $\alpha$) and
presumably also under more noisy conditions with larger values of
$\sigma$. More general epidemiological potentials could be considered
through the SDE in gradient formulation,
\begin{align}
  dP(t) &= -\nabla E(P_t) \, dt + \sigma dW(t),
\end{align}
for which the stationary measure is still given by \eqref{eq:rho}. For
general SDE models, a change of variables allows us to locally
consider an SDE with constant noise \cite{shoji1998approximation} and,
in turn, any time-homogeneous SDE with constant noise,
\begin{equation}
  dX(t) = f(X_t)dt+\sigma dW(t),
\end{equation}
is of course readily linearized into an \OU-process.

\subsection{Exact sampling}

The \OU-process \eqref{eq:OU1}--\eqref{eq:OU2} is a Gaussian
continuous process with mean and covariance given by
\begin{flalign}
  \label{eq:OUMean}
  \Expect[X_t] &= \frac{k}{\mu}(1-e^{-\mu t})+e^{-\mu t}X_0, \\
  \label{eq:OUCov}
  \Cov[X_t,X_s] &= \frac{\sigma^2}{2\mu}(e^{-\mu|t-s|}-e^{-\mu(t+s)}),
\end{flalign}
such that the stationary distribution is
$X_\infty \sim \normal(k/\mu,\sigma^2/(2\mu))$.

Numerical simulation procedures are typically based on Euler-type
discretization methods. Assume for convenience a fixed numerical time
step $\Delta t$ and compute
\begin{flalign}
  \label{eq:OUEul}
  X_{n+1}&=X_n+(k-\mu X_n)\Delta t +\sigma \Delta W_n,
\end{flalign}
with $\Delta W_n$ being i.i.d.~normally distributed numbers of zero
mean and variance $\Delta t$. The numerical trajectory $(X_n)$ is then
an approximation to the \OU-process $(X(t_n))$ at discrete times
$t_n = n\Delta t$ for $n = 0,1,\ldots$ and is also a Gaussian vector
with mean and covariance
\begin{flalign}
  \label{eq:EulMean}
  \Expect[X_n] &= \frac{k}{\mu} + (1-\mu \Delta t)^n
                 (X_0-\frac{k}{\mu}), \\
  \label{eq:EulCov}
  \Cov(X_n, X_{n+p}) &= \frac{\sigma^2 \Delta t}{1-(1-\mu \Delta t)^2}(1-(1-\mu \Delta t)^{2n})(1-\mu\Delta t)^p.
\end{flalign}
Actually, \eqref{eq:OUEul} forms an $AR(1)$-sequence
\cite{ghosh2016discrete}. Comparing
\eqref{eq:OUMean}--\eqref{eq:OUCov} with
\eqref{eq:EulMean}--\eqref{eq:EulCov} we readily find the following
useful result.

\begin{proposition}[\textit{Exact \OU-samples and backward analysis}]
  \label{prop:bwd}
  The discrete process $(X_n)$ given by the explicit Euler method
  \eqref{eq:OUEul} follows the law of the discrete samples of an exact
  \OU-process with perturbed parameters
  $(k_{\Delta t},\mu_{\Delta t},\sigma_{\Delta t})$, where
  \begin{flalign}
    \left.
    \begin{array}{rl}
      1-\mu \Delta t &= \exp{(-\Delta t \mu_{\Delta t})} \\
      \frac{k}{\mu} &= \frac{k_{\Delta t}}{\mu_{\Delta t}} \\
      \frac{\sigma^2}{2\mu -\mu^2 \Delta t} &= \frac{\sigma_{\Delta t}^2}{2 \mu_{\Delta t}}
    \end{array}
    \right\}
  \end{flalign}
  This system of equations can be solved explicitly provided
  $\Delta t < \mu^{-1}$,
  \begin{flalign}
    \label{eq:bwdparams1}
    \left.
    \begin{array}{rl}
      k_{\Delta t}&= -\frac{k\log{(1-\mu \Delta t)}}{\mu \Delta t} \\
      \mu_{\Delta t} &= -\frac{\log{(1- \mu \Delta t)}}{\Delta t} \\
      \sigma_{\Delta t}^2 &= -2\sigma^2 \frac{\log{(1-\mu \Delta t)}}{2 \mu \Delta t - \mu^2 \Delta t^2}
    \end{array}
    \right\}
  \end{flalign}
  These perturbed parameters are all first order perturbations in
  $\Delta t$ of the exact parameters.
\end{proposition}

The inverse of \eqref{eq:bwdparams1} is that
\begin{flalign}
  \label{eq:bwdparams2}
  \left.
  \begin{array}{rl}
    k &= \mu \frac{k_{\Delta t}}{\mu_{\Delta t}} \\
    \mu &= \frac{1-\exp(-\Delta t\mu_{\Delta t})}{\Delta t} \\
    \sigma^2 &= \sigma^2_{\Delta t} \, \frac{2-\Delta
                 t\mu}{2\mu_{\Delta t}} \mu
  \end{array}
  \right\}
\end{flalign}
It follows that, for \emph{given} parameters
$[k_{\Delta t},\mu_{\Delta t},\sigma_{\Delta t}]$, a forward Euler
simulation using the new set of parameters $[k,\mu,\sigma]$ defined by
\eqref{eq:bwdparams2} will produce a sample trajectory obeying the law
of the \OU-process with the given parameters exactly.



\section{Bayesian filtering with full information}
\label{sec:UnfilteredState}

We first briefly consider in this section the behavior of the Bayesian
posterior as a function of the prior and of data in the sense of full
state measurements. Hence we suppose that sampled process data
$\Ddata_\Ndata = (d_i)_{i =0}^{\Ndata}$ is available at some fixed
time step $h$, that is, $d_i = X(t_i)$ with $t_i = ih$ and $X(\cdot)$
an \OU-process with $X_0$ for simplicity considered a known sure
number.


It will be convenient to consider the following reparametrization of
the \OU-process
\begin{flalign}
  \label{eq:phi2abg}
  \left. \begin{array}{rl}
            \alpha &= \frac{k}{\mu} \in \Realdom \\
           \beta &= e^{- \mu h} \in (0,1) \\
            \gamma &= \frac{2 \mu}{\sigma^2} \cdot \frac{1}{1-e^{-2
                     \mu h}} \in \Realdom_+\\
          \end{array} \right\}
\end{flalign}
and we put $u \equiv [\alpha,\beta,\gamma]$ for brevity, also defining
$u_0 \equiv [\alpha_0,\beta_0,\gamma_0]$, i.e., the true parameters
generating the data.

Under this reparametrization we readily find from
Proposition~\ref{prop:bwd} that the \emph{exact} discrete process can
be written in the $AR(1)$-sequence form
\begin{align}
  \label{eq:AR1}
  X_{i+1} &= \beta X_i+\alpha (1-\beta)+\gamma^{-1/2} \xi_i,
\end{align}
where the $\xi_i \sim \normal(0,1)$ are independent. Writing
$\delta := \alpha(1-\beta)$, explicit least squares estimators for the
parameters can be found by solving
\begin{flalign}\label{eq:lsq1}
  \min_{\delta,\beta} \| A \times [\delta,\beta]^T-b\|_2^2 &=:
  \min_{\delta,\beta} \left\| \begin{bmatrix}
      \one & [d_i]_{i=0}^{\Ndata-1}
    \end{bmatrix}  \times [\delta,\beta]^T-
    [d_{i+1}]_{i=0}^{\Ndata-1}\right\|_2^2,
  \intertext{where the brackets over the data form column vectors. The
    residual of this solution implies the corresponding estimator for $\gamma$,}
  \label{eq:lsq2}
  \hat{\gamma}^{-1} &= (\Ndata-2)^{-1}\| A \times [\hat{\delta},\hat{\beta}]^T-b \|_2^2.
\end{flalign}
By the Gaussian character of the \OU-process these estimators coincide
with the maximum likelihood estimators.

The Bayesian convergence to the true parameters as $\Ndata \to \infty$
can be characterized by either the Bernstein von Mises theorem (BvM),
for $h$ fixed, or by contrast functions convergence for a fixed time
window $t \in [0,T]$. BvM states that the Bayesian posterior converges
to a normal distribution centered at the maximum likelihood estimate
with the inverse Fisher information matrix (FIM) as covariance
\cite{BASAWA1980255, mishra1987rate, bishwal2000rates}. Similarly, the
contrast functions approach considers the convergence of the
approximation of discretized observations, approaching the same normal
distribution for $\Ndata$ large enough \cite{florens1989approximate,
  genon1990maximnm, kessler1997estimation}. In the present case the
FIM can be determined explicitly using the log-likelihood
$\log\mathcal{L}(X_{i+1} | X_i; \; u)$ induced by \eqref{eq:AR1} and
the definition
\begin{align}
  \Sigma &\equiv -\Expect \left[ \frac{\partial^2}{\partial
           u^2} \log\mathcal{L} \right](u_0). \\
  \intertext{The result is that}
  \label{eq:Sigma_val}
  \Sigma &= \diag\left((1-\beta_0)^{2}\gamma_0, 1/(1-\beta_0^2),
  1/(2\gamma_0^{2})\right).
\end{align}

\begin{proposition}[\textit{BvM Theorem for \eqref{eq:AR1}}]
  Provided that the prior has a continuous and positive density in an
  open neighborhood of $(\alpha,\beta,\gamma)$, as
  $\Ndata \to \infty$ we have that the ML-estimators
  \eqref{eq:lsq1}--\eqref{eq:lsq2} converge to the true values as
  \begin{align}
    \label{eq:fullstateQ}
    \sqrt{N}(\hat{u}-u_0) \xrightarrow{d} \normal(0,\Sigma^{-1}),
  \end{align}
  (convergence in the sense of distribution) with $\Sigma$ defined by
  \eqref{eq:Sigma_val}. Alternatively, \eqref{eq:fullstateQ} dictates
  the asymptotic convergence of the mean of the Bayesian posterior.
\end{proposition}

To summarize, the asymptotic variances of all parameters are
independent of $\alpha_0$. The noise term $\gamma_0$ mainly affects
the convergence of $\alpha_0$ (and $\gamma_0$ in an absolute sense).
Finally, an increase of $\beta_0$ implies a faster convergence towards
$\beta_0$ at the cost of a slower convergence towards $\alpha_0$.
Under the interpretation of an SIS-model
\eqref{eq:SIS2OU}--\eqref{eq:SIS2OU_R0}, the parameter $\alpha_0$ is
in one-to-one correspondence with the basic reproduction number $R_0$,
hence its central interest here.


The asymptotic nature of both the BvM Theorem and the contrast
function convergence is a poor match in epidemiological situations
where data is often scarce and poorly informative at the time scale
over which the parameters can be considered static. This motivates our
interest in also the pre-asymptotic regime of the Bayesian
posterior. We therefore consider prior densities of the specific form
\begin{equation}
  \label{eq:prior_form}
  \pi(u)\propto\gamma^r\exp{\left(-\frac{\gamma}{2}P(\alpha,\beta)\right)},
\end{equation}
where, for integrability, $P$ is to be a polynomial of degree $2$ in
$\alpha$ and nonnegative for $\beta\in(0,1)$. This choice has the
convenient property that the posterior measure after $\Ndata$
observations is
\begin{flalign}
  \label{eq:posterior_form}
  \Pi^{(\Ndata)}(u) &\propto \gamma^{\frac{\Ndata}{2}+r}
  \exp{\left(-\frac{\gamma}{2}(Q_\Ndata+P)(\alpha,\beta)\right)}, \\
  \intertext{where}
  \label{eq:Qdef}
  Q_\Ndata(\alpha,\beta) &\equiv \sum_{i=0}^{\Ndata-1} (d_{i+1}-\beta d_i-
  \alpha(1-\beta))^2,
\end{flalign}
and where $d_i = X(t_i)$ is the observation at time $t_i$.

\begin{remark}
  To include also the case of flat priors, while avoiding
  technicalities for nonintegrable densities, we note that, since the
  value of a constant prior has no influence on the posterior, we can
  still determine a posterior after a single initial observation, and
  then use this posterior as a prior for the rest of the observations.
\end{remark}

The following result examines the convergence of the Bayesian
posterior.
\begin{theorem}[\textit{Convergence of log-likelihood}]
  \label{th:Qconvergence}
  Consider the polynomial $q_\Ndata$:
  \begin{equation}
    q_\Ndata(\alpha,\beta) \equiv \Ndata^{-1}Q_\Ndata(\alpha,\beta) =
    \Ndata^{-1} \sum_{i=0}^{\Ndata-1}(d_{i+1}-\beta d_i -\alpha(1-\beta))^2.
  \end{equation}
  Then as $\Ndata \to \infty$ we have the almost everywhere uniform
  convergence on every compact of $q_\Ndata$ towards the function $f$,
  \begin{equation}
    f(\alpha,\beta) =
    (1-\beta)^2\left[\frac{1}{\gamma_0(1-\beta_0^2)}+
      (\alpha_0-\alpha)^2\right]+\frac{2\beta}{\gamma_0(1+\beta_0)}.
  \end{equation}
\end{theorem}

\begin{proof}
  For a general point $(\alpha,\beta)$, we rely on the ergodic theory
  of Markov chains \cite[p.~472]{athreya2006measure} to get that
  $\Ndata^{-1}\sum g(d_i,d_{i+1})$ converges almost surely towards
  $\Expect[g(d_i,d_{i+1})]$ for $g$ integrable against the stationary
  measure. Using \eqref{eq:OUMean}--\eqref{eq:OUCov} this implies the
  limits:
  \begin{flalign}
    \Ndata^{-1}\sum_{i=0}^{\Ndata-1} d_{i+1}-\beta d_i
    &\longrightarrow \alpha_0(1-\beta), \\
    \Ndata^{-1}\sum_{i=0}^{\Ndata-1} (d_{i+1}-\beta d_i)^2
    &\longrightarrow (\alpha_0(1-\beta))^2+
    \frac{1+\beta^2}{\gamma_0(1-\beta_0^2)}-
    \frac{2\beta\beta_0}{\gamma_0(1-\beta_0^2)}.
  \end{flalign}
  Combined, we obtain the claimed limit in a pointwise sense. As each
  $(q_N)$ is a polynomial in $(\alpha, \beta)$ of bounded degree,
  pointwise convergence is equivalent to convergence of the
  coefficients, and hence implies the convergence on all compact sets.
\end{proof}

An obvious extension is to consider measurements polluted by noise,
say, $d_i = X_i+\eta_i$, for i.i.d.~$\eta_i \sim \normal(0,\eta)$ and
some variance $\eta$. The posterior so obtained is readily computed
via a recursive Kalman filter but does not have a simple analytic
form. To first order in $\eta$, however,
\eqref{eq:prior_form}--\eqref{eq:posterior_form} still holds provided
$\gamma$ is replaced with $[1/\gamma+\eta(1+\beta^2)]^{-1}$ and
$Q_N$ in \eqref{eq:Qdef} is replaced with
\begin{align}
  Q'_\Ndata&\equiv \sum_{i=0}^{\Ndata-1} \left[ d_{i+1}-\beta d_{i}-
  \alpha(1-\beta)+\eta\beta\gamma(d_{i}-\beta d_{i-1}-
  \alpha(1-\beta)) \right]^2,
\end{align}
where the right $\eta\beta\gamma$-term is skipped when $i =
0$. Intuitively, the first order effect of noise in data is to broaden
the posterior with the variance of this noise. We next proceed to
investigate more severe truncations of the measurements from the
epidemiological process.

\section{Bayesian filtering of surveillance data}
\label{sec:FilteredState}

In the previous section we considered full process data to be
available without any extrinsic noise. This models the best possible
Bayesian setup but is also clearly unrealistic in epidemics. As a
model of a more challenging situation we consider in this section the
recorded data to be some (possibly stochastic) function of the
\OU-process $X(\cdot)$, which itself is considered a latent
variable. We are initially concerned with \emph{binary data} of the
form $\Ddata_\Ndata = (d_i)_{i=0}^\Ndata$ where
$d_i = Y_i = \one_{X(t_i) \ge c}$, again over a uniform grid in time
$t_i = ih$, and for a known \emph{filter} cut-off value $c$ (see also
the related setup in \cite{clippedOU}). That is, the epidemiological
interpretation is that the data is considered to be time-discrete
information about whether or not the prevalence $X(\cdot)$ of the
population is above or below a certain known threshold $c$. With the
prevalence process hidden one is forced to estimate it simultaneously
with any parameter estimates. Using the $AR(1)$-form \eqref{eq:AR1} we
have that the stationary measure for the $(p+1)$ steps
$(X_i,\ldots,X_{i+p})$ is Gaussian $\normal(\alpha,\Sigma)$ with
$\Sigma$ given by (for $p \ge 0$)
\begin{flalign}
  \label{eq:equations}
  \Sigma = \Sigma(\beta,\gamma) = \frac{\gamma^{-1}}{1-\beta^2}
  \begin{bmatrix}
    1&\beta&\beta^2&\cdots&\beta^p \\
    \beta&1&\beta&\cdots &\beta^{p-1} \\
    \beta^2&\ddots&\ddots & \ddots &\cdots \\
    \vdots&\ddots &\beta&1&\beta \\
    \beta^p&\cdots&\beta^2&\beta&1
  \end{bmatrix}.
\end{flalign}
In order to be able to conduct an analysis where information is
obtained only from the observable $Y(\cdot)$, we take
\eqref{eq:equations} as a motivation for the following:

\begin{assumption}[\textit{$p$-step Markovian stationary assumption}]
  \label{ass:stationary}
  We say that we work under the stationary assumption whenever we
  assume that the law of the latent variable $Y(\cdot)$ can be
  directly inferred from the stationary $p$-step law \eqref{eq:equations}.
\end{assumption}

We stress that it is known that this type of clipped Gaussian
processes are not $p$-step Markov for any $p$
\cite{NotpMarkov}. Assumption~\ref{ass:stationary} rather serve as an
approximation where we approximately model $Y(\cdot)$ as if it was
$p$-step Markov with law deduced from \eqref{eq:equations}. The
pseudo-likelihood for the filtered variable $Y$ then becomes
\begin{align}
  \label{eq:pseudo_likelihood}
  \mathcal{L}_\Ndata(u) &= \prod_{i=0}^{\Ndata-p} \varphi_u(Y_i,\ldots,Y_{i+p})
\end{align}
where $\varphi_u(e)$ denotes the probability for a Gaussian stationary
filtered process with parameters $u$ to be $e\in\{0,1\}^{p+1}$. We
also define the pseudo-potential
\begin{equation}
  \label{eq:pseudo_potential}
  \Phi_\Ndata^u\equiv
  \log\mathcal{L}_N(u) = \sum_{i=0}^{\Ndata-p}\log\varphi_u(Y_{i},\ldots,Y_{i+p}).
\end{equation}

We show below in \S\ref{subsec:binary} that the stationary assumption
allows for converging posterior estimates for the parameter $\beta$,
but leaves any prior density unchanged over a certain curve in
$(\alpha,\gamma)$. In \S\ref{subsec:uncertainc} we briefly consider
the filter cutoff value $c$ (the sensitivity) to be uncertain and
straightforwardly show that any prior density on $c$ is unaffected by
data. We next replace the sharp deterministic filter by a stochastic
filter implementing a sigmoidal response function and sharpen our
results in this more general setting in
\S\ref{subsec:stoch_binary}. Finally, in \S\ref{subsec:nonbinary} we
consider slightly more informative measurements consisting of a finite
discrete response and we show that this resolves the singularity
issues associated with purely binary measurements.

\subsection{Binary measurements}
\label{subsec:binary}

We first show that under the $p$-fold stationary assumption one can
estimate the correlation term $\beta$ rather well, but increasing the
gap between the filter threshold $c$ and the mean $\alpha$ has an
effect on the likelihood which is indistinguishable in law from
increasing the noise $\gamma$.

We start with a technical lemma.
\begin{lemma}[\textit{Equality in law}]
  \label{lem:eqlaw}
  For $u$ in any admissible set of parameters, let $p_u$ be the law of
  $(Y_0,\ldots,Y_p)$ under Assumption~\ref{ass:stationary}
  corresponding to $\normal(\alpha_u,\Sigma_u)$. Then $p_u= p_{w}$ if
  and only if $\beta_u=\beta_{w}$ and
  $\sqrt{\gamma_u}(c-\alpha_u)=\sqrt{\gamma_{w}}(c-\alpha_{w})$.
\end{lemma}

\begin{proof}
  $(\Longrightarrow)$ Suppose $p_u=p_{w}$.  Consider
  $(X_0^u,\ldots,X_p^u)$ a Gaussian vector with law
  $\normal(\alpha_u,\Sigma_u)$.  As
  $\Pr(X_0^u\leq c)=\sum_{e \in \{0,1\}^{p}}\varphi(0,e)$, we deduce
  that $\Pr(X_0^u\leq c)=\Pr(X_0^{w}\leq c)$. This can be written as
  an equality between standard cumulative distribution functions of
  Gaussians, since $\sqrt{\gamma_u(1-\beta_u^2)}(X_0^u-\alpha_u)$ and
  $\sqrt{\gamma_w(1-\beta_w^2)}(X_0^w-\alpha_w)$ are standard
  Gaussians.  This translates into:
  \begin{equation}
  \sqrt{\gamma_u(1-\beta_u^2)}(c-\alpha_u)=
  \sqrt{\gamma_{w}(1-\beta_w^2)}(c-\alpha_{w})
  \end{equation}
  We also have $\Pr(X_0^u,X_1^u\leq c)=\Pr(X_0^{w},X_1^{w}\leq c)$, and
  this implies, with $A = \sqrt{\gamma_u(1-\beta_u^2)}(c-\alpha_u)$,
  \begin{multline}
    \frac{1}{\sqrt{1-\beta_u^2}}
    \int\limits_{(-\infty,A]^2}\exp\left(-\frac{x^2+y^2-2\beta_u
        xy}{\sqrt{1-\beta_u^2}}\right) \, dx dy\\
    =\frac{1}{\sqrt{1-\beta_{w}^2}}
    \int\limits_{(-\infty,A]^2}\exp\left(-\frac{x^2+y^2-2\beta_{w}
        xy}{\sqrt{1-\beta_{w}^2}}\right) \, dx dy.
  \end{multline}
  This last expression shows that the map
  $\beta_u\mapsto\Pr(X_0^u,X_1^u\leq c)$ is locally analytic.  As a
  corollary of Slepian's lemma \cite{joag1983association}, this map is
  also increasing and hence strictly increasing.  It follows that the
  last equality implies $\beta_u=\beta_w$.

  $(\Longleftarrow)$ The law $p_u$ is uniquely determined by its
  $2^{p+1}$ values. All these values are of the type
  \begin{equation}
    I(Q;\; u) =\int_{Q(A)}
    (\det M(\beta_u))^{-(p+1)/2}  \times
    \exp\left(-x M(\beta_u)x^T\right) \, dx,
  \end{equation}
  where $Q(A)$ is a product of $p+1$ intervals, each being either
  $(-\infty, A]$ or $[A,+\infty)$. And so, if $\beta_u = \beta_w$ and
  $\sqrt{\gamma_u}(c-\alpha_u) = \sqrt{\gamma_{w}}(c-\alpha_{w})$, the
  laws $p_u$ and $p_{w}$ are indeed equal.
\end{proof}

\begin{theorem}[\textit{Non-identifiability}]
  \label{th:nonidentifiability}
  Assume $p\geq 1$.
  For any integrable prior $\pi$, non-zero on $\mathcal{E}$, and for
  $f$ bounded, as $\Ndata \to \infty$,
  \begin{equation}
    \int f\mathcal{L}_Nd\pi
    \rightarrow \int_\mathcal{E}fd\tilde{\pi},
  \end{equation}
  where $\mathcal{E} \equiv \{u;\; \sqrt\gamma(c-\alpha) =
  \sqrt\gamma_0(c-\alpha_0),\, \beta = \beta_0\}$ and $\tilde{\pi}$ is
  the prior restricted to the set $\mathcal{E}$.
\end{theorem}

\begin{proof}
  At first, thanks to Lemma~\ref{lem:eqlaw}, the law
  $\varphi(Y_i\ldots,Y_{i+p})$ is not characterized by
  $u_0 = [\alpha_0,\beta_0, \gamma_0]$, but rather by
  $[\sqrt{\gamma_0}(c-\alpha_0),\beta_0]$.  As the $Y_i$'s are binary,
  the quantity $\varphi(Y_i,\ldots,Y_{i+p})$ can only take $2^{p+1}$
  values, which we denote by $\varphi_e$ for $e\in\{0,1\}^{p+1}$. This
  means that the log-likelihood under Assumption~\ref{ass:stationary}
  can be written as a finite sum:
  \begin{equation}
    \log\mathcal L_N = \sum\limits_{e\in\{0,1\}^{p+1}}
    N(e) \log\varphi_e.
  \end{equation}
  where
  \begin{equation}
    N(e) = \sharp\{i\in\{0,\Ndata-p\}; \; (Y_i,\ldots,Y_{i+p})=e\}
  \end{equation}

  Now, as the discrete \OU-process is an $AR(1)$-sequence
  \cite{ghosh2016discrete}, we may infer the convergence of
  $N(e)/\Ndata$ towards a stationary value:
  \begin{equation}
    \frac{\log\mathcal{L}_N}{\Ndata} \rightarrow
    \sum\limits_{e\in\{0,1\}^{p+1}}\varphi_e^{u_0}\log\varphi_e^u,
  \end{equation}
  where $\varphi_e^u$ denotes the law of a stationary filtered process
  with parameter $u$.  We note that
  \begin{equation}
    \lim_{\Ndata \to \infty} \frac{\log\mathcal{L}_\Ndata}{\Ndata}=\sum\limits_{e\in\{0,1\}^{p+1}}\varphi_e^{u_0}
    \log\frac{\varphi_e^u}{\varphi_e^{u_0}}+
    \sum\limits_{e\in\{0,1\}^{p+1}}\varphi_e^{u_0}\log\varphi_e^{u_0}.
  \end{equation}
  Up to a constant this limit is the negative of the Kullback-Leibler
  divergence between the laws $\varphi^u$ and $\varphi^{u_0}$, which
  vanishes if and only if the two distributions are equal. We already
  know that the laws of $\varphi^u$ and $\varphi^{u_0}$ are equal if
  and only if $\beta = \beta_0$ and
  $\sqrt{\gamma}(c-\alpha) = \sqrt{\gamma_0}(c-\alpha_0)$, i.e., if
  $u \in \mathcal E$.  Consider the associated relation:
  \begin{equation}
    u = (\alpha,\beta,\gamma)\sim u'=(\alpha',\beta',\gamma') \text{ iff }
    \beta=\beta' \text{ and }\sqrt{\gamma}(c-\alpha)=
    \sqrt{\gamma'}(c-\alpha').
  \end{equation}
  It is straightforward to show that this relation is reflexive,
  symmetric, and transitive, thus forming an equivalence
  relation. Thanks to Lemma~\ref{lem:eqlaw}, we know that the map
  $u\mapsto\mathcal{L}_\Ndata(u)$ factorizes through the equivalence
  relation to a map $w\mapsto\tilde{\mathcal{L}}_\Ndata(w)$ where $w$
  ranges over the different equivalence classes of $u$.

  Lemma~\ref{lem:eqlaw} also allows us to state that the map
  $\tilde{\mathcal{L}_\Ndata}$ is injective. Consider $w_0$ in the
  equivalence class of $u_0$. Then for any neighborhood $W$ of $w_0$,
  there is a $\delta > 0$ such that for $\Ndata$ large enough,
  $w \not \in W$ implies
  \begin{equation}
    \frac{\log{\tilde{\mathcal{L}}_\Ndata(w_0)}}{\Ndata} \ge
    \frac{\log{\tilde{\mathcal{L}}_\Ndata(w)}}{\Ndata} + \delta,
  \end{equation}
  which means that the posterior measure of $W$ converges towards
  1. Since the factorized map
  $\tilde{\mathcal{L}}_\Ndata \rightarrow \delta_{w_0}$, the theorem
  follows.
\end{proof}

The same argument implies that, if we consider the case $p = 0$, the
posterior will be even more degenerate. For $p = 0$ one gets the
convergence
\begin{equation}
  \int f\mathcal{L}_N d\pi\rightarrow
  \int_{\mathcal{E}'}fd\tilde\pi,
\end{equation}
where now $\tilde\pi$ is the prior restricted to $\mathcal{E}'$:
\begin{equation}
  \mathcal{E'} \equiv \{u;\; \sqrt{\gamma_u(1-\beta_u^2)}(c-\alpha_u)=
  \sqrt{\gamma_0(1-\beta_0^2)}(c-\alpha_0)\}.
\end{equation}
This case is relevant whenever we consider large gaps of time in
between the measurements such that they can be considered practically
independent.

At this point, let us remind that the quantities
$\mathcal{L}_\Ndata(u)$ and $\Phi_\Ndata^p(u)$ are defined in
equations~\eqref{eq:pseudo_likelihood}--\eqref{eq:pseudo_potential}.

\begin{theorem}[\textit{Rate of convergence of pseudo-potential}]
  Under the $p$-Markovian stationary Assumption~\ref{ass:stationary}
  there exists a mapping $f_p$ and a constant $\sigma_p>0$ such that,
  as $\Ndata \to \infty$, we have the convergence in law:
  \begin{equation}
    \sqrt{N}\bigg(\frac{\Phi_N^{p}(u)}{N}-f_p(u)\bigg)\rightarrow
    \normal(0,\sigma_p).
  \end{equation}
\end{theorem}

\begin{proof}
  This is a direct consequence of the Central Limit Theorem, with
  $f_p(u)$ being the mean value of $\Phi_N^p(u)$ and $\sigma_p$ a
  nonnegative constant.
\end{proof}

\subsection{Propagation of filter uncertainty}
\label{subsec:uncertainc}

A relevant variation of the theme is to consider the parameter $c$
(the ``test sensitivity'') an uncertain parameter. Mathematically,
this means considering the likelihood $\mathcal L_N$ a function of
$(\alpha,\beta,\gamma,c)$, and priors and posteriors depending also on
$c$. However, the following result shows that this setup will not
produce any more information about the parameters.

\begin{theorem}[\textit{Translation of filter uncertainty}]
  Consider a prior of the form $\mu(\alpha)\pi(c-\alpha,\beta,\gamma)$
  and the pseudo-likelihood in \S\ref{subsec:binary}. Then the
  resulting posterior measure will be of the form
  $\mu(\alpha)\Pi(c-\alpha,\beta,\gamma)$, where, for fixed $c$,
  $\Pi(c-\cdot,\cdot,\cdot)$ is the pseudo-posterior from the prior
  $\pi(c-\cdot,\cdot,\cdot)$.
\end{theorem}

\begin{proof}
  This result is immediate once one realizes that the
  pseudo-likelihood has the property that, for any $t\in \Realdom$,
  \begin{equation}
    \mathcal L_N(\alpha,\beta,\gamma,c) =
    \mathcal L_N(\alpha+t,\beta,\gamma,c+t).
  \end{equation}
\end{proof}

In other words, the only information we can infer on the parameters
$(c,\alpha)$ is the gap $c-\alpha$. Any uncertainty of $c$ can be
understood as an uncertainty on $\alpha$, and vice-versa.

\subsection{Non-perfect binary measurements}
\label{subsec:stoch_binary}

Rather than a sharp cut-off value $c$, most environmental sampling
methods obey some kind of \emph{sensitivity response}, e.g., of
sigmoid character, with a quick rise in the detection probability as
one progresses through some threshold region. Examples here could
include sampling and subsequent analysis of sewage water or animal
droppings, but this would also be a relevant model in the case of
statistical regression estimates using data obtained via
self-reporting smartphones applications.

A general ansatz to capturing this situation is to consider
\begin{equation}
  Y_i\sim \mathcal{B}(s(X_i)),
\end{equation}
where $s$ is a map from $\Realdom$ to $[0,1]$ and $\mathcal B$ denotes
the Bernoulli law. Typically, the map $s$ is sigmoidal with a gradient
around the threshold $c$ which depends on the sensitivity and
specificity of the test. We naturally ask that the efficiency of the
filter does not depend on its previous use, i.e., that $(Y_i|X_i)_i$
is an independent family. To construct such an object, one may
consider an i.i.d.~family of uniformly distributed variables $(\xi_i)$
on $(0,1)$, independent from $(X_i)$, and set
\begin{equation}
  Y_i = \one_{\xi_i\leq s(X_i)}.
\end{equation}
This object clearly fulfills all properties mentioned and is general
enough to capture also quite specific situations.

We will now establish our result in this more general
setting. Consider a bounded and continuous map
$f:\Realdom^{p+1}\rightarrow\Realdom$.  We observe the \OU-process
$X(t)$ through the map $f$, so the observations are
$Y_i = f(X_i,\ldots,X_{i+p})$ where as before $X_i = X(t_i)$. To
define a pseudo-likelihood, we reason as if the observed data can be
regarded as stationary.  We thus denote by $f_u$ the law of a
\emph{stationary} \OU-process with parameters $u$ filtered via the
measurement map $f$ and we remind ourselves that the latent process
$X(t)$ is not necessarily stationary.

We consider a likelihood of the kind
\begin{equation}
  \mathcal{L}_N \propto \exp\left(\sum_{i=0}^{\Ndata-p}
    \log f_u(Y_i)\right),
\end{equation}
normalized to mass one, and we denote the potential by
\begin{equation}
  g_\Ndata(u) := \Ndata^{-1}\sum\limits_{i=0}^{\Ndata-p}
  \log f_u(Y_i).
\end{equation}
This setup simply corresponds to observing the data portion
$(X_i,\ldots,X_{i+p})$ via the filter $f$ and, assuming stationarity,
building a pseudo-likelihood. Unfortunately this does not directly
include the intended case of a sigmoid filter, as this would rather
involve measurements $h(\xi_i,X_i,\ldots,\xi_{i+p},X_{i+p})$ with
$(\xi_i)$ an i.i.d.~sequence independent from $(X_i)$ and uniform on
$(0,1)$. However, once the proposed case is examined, one can get to
the latter case by exchanging the order of integration using the
Fubini property, i.e., studying the behavior of
$h(t_0,X_i,\ldots,t_p,X_{i+p})$ and then integrate over
$(t_0,\ldots,t_p)$. For example, consider the case $p = 1$. Then the
intended map $h$ would be
\begin{equation}
  h(t_0,x_0,t_1,x_1)=(\one_{t_0\leq s(x_0)},\one_{t_1\leq s(x_1)})
\end{equation}
For a fixed $u$, $f_u$ can take at most $4$ values, e.g.,
\begin{equation}
  f_u(1,1) = \Expect[s(Z_0)s(Z_1)],
\end{equation}
where $(Z_0,Z_1)$ is a stationary \OU-process of parameter $u$, and
similarly for $f_u(0,1)$, $f_u(1,0)$, and $f_u(0,0)$. As this
technique would render the proof lengthy, we decided to put it aside.

In order to state our result, we need to specify some minimal set of
regularity conditions. Convergence of the potential is required as
well as definiteness in the Kullback-Leibler divergence
($\KL$). Additionally, we shall also require a separation condition in
the large data limit.
\begin{assumption}[\textit{Regularity}]
  \label{ass:hyp}
  We assume the following specifics:
  \begin{enumerate}
  \item \label{it:ass1} The potential $g_N$ converges uniformly on the
    compacts of $u$.
  \item \label{it:ass2}
    $\KL(f_{u_0},f_u) = 0 \Longrightarrow u = u_0$.
  \item \label{it:ass3} There is a $\delta > 0$ and a compact
    neighborhood $K$ of $u_0$ such that, for $\Ndata$ large enough and
    for $u \in K^\complement$, $g_N(u_0) > g_N(u)+\delta$.
  \end{enumerate}
\end{assumption}

\begin{theorem}[\textit{Weak convergence}]
  \label{th:genweak}
  Under Assumption~\ref{ass:hyp}, we have the weak convergence
  \begin{equation}
    \mathcal{L}_N\rightarrow\delta_{u_0}.
  \end{equation}
\end{theorem}

This theorem can be adapted to cover other situations:
\begin{itemize}
\item If $\KL(f_{u_0},f_u) = 0$ does not imply $u = u_0$, one can try
  to factorize the map through an equivalence relation (as in the
  proof of Theorem~\ref{th:nonidentifiability}) to get the convergence
  towards the indicative function of the set
  $\{u; \; \KL(f_{u_0},f_u) = 0\}$.
\item If we only want to consider parameters within a subset $U$ of
  the set of parameters, one can always consider the assumptions
  restricted to the set $U$, and then one would have the convergence
  only for priors with support in $U$.
\end{itemize}

\begin{proof}
  We divide the proof in two steps as follows.

  \paragraph{Step 1} Let us call $g$ the limit of $(g_N)$, which by
  assumption is not random. The same argument as in the proof of
  Theorem~\ref{th:Qconvergence}, the ergodicity of the process $(X_t)$,
  allows us to state that the limit
  function must be $\Expect[\log f_u]$ (up to a constant, the limit is
  actually $-\KL(f_{u_0},f_u)$). As adding a constant is equivalent to
  multiplying all posteriors by a nonnegative constant, we may assume
  the limit function to be equal to $u \mapsto -\KL(f_{u_0},f_u)$.

  \paragraph{Step 2} As $\mathcal{L}_N$ is proportional to
  $\exp(g_N)$, the assumptions allow us to deduce that the mass of any
  neighborhood of $u_0$ converges towards $1$. This implies the weak
  convergence of the posteriors towards the Dirac $\delta_{u_0}$.
\end{proof}

\subsection{Non-binary measurements}
\label{subsec:nonbinary}

As a final variation on the theme we now show how
Theorem~\ref{th:genweak} can be applied in such a way as to overcome
the issues with the non-identifiability of
Theorem~\ref{th:nonidentifiability}. The idea is that the test with
one filter $c$ evidently at best gives us a curve containing the
parameter $u_0$, namely the one satisfying $\beta = \beta_0$, and
$\sqrt{\gamma}(c-\alpha) = \sqrt{\gamma_0}(c-\alpha_0)$. Suppose
instead that we have two kinds of tests, one with cut-off $c_1$ and
one with cut-off $c_2$. Since the intersection of the two curves
implied by the respective filters $c_1$ and $c_2$ is exactly the point
$u_0$, it is natural to assume that two filters are enough to get the
convergence of the posterior towards $\delta_{u_0}$. We show that this
is indeed the case.
\begin{theorem}[\textit{Trinary filter}]
  \label{th:doublefilter}
  Let $c_1<c_2$ and consider the filtered values $Y_i$ to be
  \begin{equation}
    Y_i = \one_{X_i>c_1}+\one_{X_i>c_2},
  \end{equation}
  and the associated pseudo-likelihood
  \begin{equation}
    \mathcal{L}_\Ndata = \prod_{i=0}^{\Ndata-p} \varphi(Z_{i},\ldots,Z_{i+p}),
  \end{equation}
  where $\varphi$ is the multivariate cumulative distribution function
  of a Gaussian $\normal(\alpha,\Sigma)$ \eqref{eq:equations}.
  Then the sequence of pseudo-likelihoods converges weakly:
  \begin{equation}
    \mathcal{L}_N\rightarrow\delta_{u_0}.
  \end{equation}
\end{theorem}

\begin{proof}
  It is sufficient to check that the sequence $(Y_i)$ verifies
  Assumption~\ref{ass:hyp}.  Considering
  $g_\Ndata := \Ndata^{-1} \log{\mathcal{L}_\Ndata}$, one has
  \begin{equation}
    g_\Ndata = \sum\limits_{e\in\{0,1,2\}^p}\frac{N_e}{\Ndata}\log{p_u(e)},
  \end{equation}
  where $N_e$ is the number of times $(Y_i,\ldots,Y_{i+p})$ matches
  $e$, and $p_u$ is the cumulative distribution function of the
  underlying Gaussian process $\normal(\alpha,\Sigma)$. From this we
  get directly conditions \eqref{it:ass1} and \eqref{it:ass3} of
  Assumption~\ref{ass:hyp}. For condition \eqref{it:ass2}, if the
  Kullback-Leibler divergence is zero, then we must have
  $p_u(e) = p_{u_0}(e)$ for all $e$.  This implies the equalities
  \begin{align}
    \left. \begin{array}{rl}
             \beta &= \beta_0 \\
             \sqrt{\gamma}(c_1-\alpha) &= \sqrt{\gamma_0}(c_1-\alpha_0)
             \\
             \sqrt{\gamma}(c_2-\alpha) &= \sqrt{\gamma_0}(c_2-\alpha_0)
           \end{array} \right\}
  \end{align}
  As $c_1\neq c_2$ we have $u=u_0$, and so the assumptions are
  fulfilled.
\end{proof}


\section{Illustrations}
\label{sec:Illustrations}

We devote this section to some illustrations of selected results from
\S\S\ref{sec:UnfilteredState} and \ref{sec:FilteredState}. We shall do
this in the intended epidemiological setting and thus no longer assume
the \OU-process, but rather the SIS- and SIS$_{\text{E}}$-models from
\S\ref{sec:SDEmeta} in the form of continuous-time Markov chains over
a discrete state-space. In \S\ref{subsec:asymptotic_uncertainty} we
investigate the precision of the predicted posterior uncertainty under
full state measurements, while in \S\ref{subsec:limits_convergence} we
offer a demonstration of the singular behavior under filtered
measurements. Finally, in \S\ref{subsec:network_epidemics} we
highlight the use of synthetic data when approaching more realistic
problems defined over a network.

The software for the numerical experiments is available for download
via the corresponding author's web-page\footnote{\updated{Refer to the
    BISDE-code at
    \href{https://user.it.uu.se/~stefane/freeware.html}{https://user.it.uu.se/$\sim$stefane/freeware.html}}}.

\subsection{Asymptotic uncertainty}
\label{subsec:asymptotic_uncertainty}

We first consider the Bayesian uncertainty under accurate measurements
and take the continuous-time Markov chain version of the SIS-model
\eqref{eq:ctmc} as an example. Using the
$\text{SIS} \leftrightarrow$~\OU\ approximate interpretation
\eqref{eq:SIS2OU} we have from \eqref{eq:phi2abg} the relations
\begin{align}
  \label{eq:map}
  &\left. \begin{array}{rl}
            \alpha_{\text{\OU}} &= \Sigma (1-R_0^{-1}) \\
            \beta_{\text{\OU}} &= \exp(-\gamma h(R_0-1)) \\
            \gamma_{\text{\OU}} &= \Sigma^{-1}R_0 \cdot [1-\beta_{\text{\OU}}^2]^{-1}
          \end{array} \right\} \quad
                                  \left. \begin{array}{rl}
            R_0 &= 1+\alpha_{\text{\OU}}(1-\beta_{\text{\OU}}^2)
                  \gamma_{\text{\OU}} \\
            \Sigma &= \alpha_{\text{\OU}}+
                     [(1-\beta_{\text{\OU}}^2)\gamma_{\text{\OU}}]^{-1} \\
            \gamma &= -\log \beta_{\text{\OU}} \cdot [h(R_0-1)]^{-1}
          \end{array} \right\}
\end{align}
where $\{R_0,\Sigma,\gamma\}$ are the SIS-model parameters.

We generate synthetic data from the Markov chain as illustrated in
Fig.~\ref{fig:SIS_illu} for $\Sigma = 1000$ and with ranges of values
$\gamma \in [0.1,1]$ and $R_0 \in (1,3.5]$. We let
$I(0) = 0.01\times\Sigma$ and sample exact values of $I(t_i)$ for
$t_i = ih$, $i = 1,\ldots,\Ndata$, and $h \equiv 1$. This corresponds
to 100 perfect samples in a closed population at a rate equivalent to
between one to one tenth the disease period unit ($= 1/\gamma$).

We evaluate the posterior over a grid of values in the
$(R_0,\gamma)$-plane by simply normalizing the likelihood for the
Markov chain given the synthetic data discussed previously. The
likelihood of the Markov chain is formally obtained by solving the
associated master equation which, however, is inconvenient to do
except for small populations $\Sigma$. A more general approach is via
a local linear Gaussian approximation and a Kalman filter. Put
$I_0 = I(0)$ and define
\begin{align}
  \nonumber
  I_{k+1} &= \left(1+\Delta t
            \beta\Sigma^{-1}(\Sigma-I_k)-\Delta t\gamma\right)I_k+w_k, \\
  \label{eq:Kalman}
  w_k &\sim \normal(0, [\Delta t\beta\Sigma^{-1}(\Sigma-I_k)+\Delta t\gamma]I_k),
\end{align}
that is, this is the forward Euler discretization of the Langevin
equations approximating the Markov chain. The Kalman filter associated
with \eqref{eq:Kalman} computes a likelihood for each data point,
albeit for a perturbed model. The relative error in the Langevin
approximation generally scales with the inverse of the population size
$\Sigma$, and can be expected to be rather small in the present
context (see \cite[Ch.~11.3]{Markovappr}). Further,
Proposition~\ref{prop:bwd} suggests analyzing the Euler discretization
via backward analysis as a parameter perturbation, but unfortunately
this is not generalizable to non-additive noise
\cite{bwdAnalSDEs}. For a resolved discretization, however,
$\gamma \Delta t \ll 1$, weak first order convergence can be expected
under broad conditions. We use the constant Kalman resolution
$\Delta t = h/4$ and next focus on the estimation error.

We have already evaluated the asymptotic covariance matrix under
accurate data in \eqref{eq:fullstateQ}. Using the linear uncertainty
transformation $Q' \approx JQJ^T$, where $J$ is the Jacobian of the
parameter map \eqref{eq:map}, and where $Q$ is the (diagonal)
covariance matrix in \eqref{eq:fullstateQ}, we can estimate the
posterior variance
\begin{align}
  \label{eq:VarR0}
  \text{Var}(R_0) &\approx \frac{R_0^3}{\Sigma \Ndata} \times
                    \frac{1+\beta_{\text{\OU}}}{1-\beta_{\text{\OU}}}.
\end{align}
A similar formula can be worked out for the variance of $\gamma$ as
well, although a bit more involved. For small enough $h$, the
denominator $1-\beta_{\text{\OU}} \sim \gamma h (R_0-1)$ in
\eqref{eq:VarR0}, and so the \emph{relative} uncertainty in any
consistent estimator of $R_0$ can be expected to depend weakly on
$R_0$ itself. This effect is seen for $R_0 > 1$ in
Fig.~\ref{fig:SIS_marginalR0} (\textit{top}), where it can also be
seen that the approximation \eqref{eq:VarR0} derived from the
\OU-approximation is somewhat optimistic. Similarly, we find that the
relative uncertainty of $R_0$ goes down with increasing values of
$\gamma$, or, which by \eqref{eq:map} is the same thing, with
decreasing correlation $\beta_{\text{\OU}}$
(cf.~Fig.~\ref{fig:SIS_marginalR0}, \textit{bottom}).

\begin{figure}[t]
  \centering
  \includegraphics[trim = 4.5cm 12.24cm 5.4cm 12.55cm,clip =
  true]{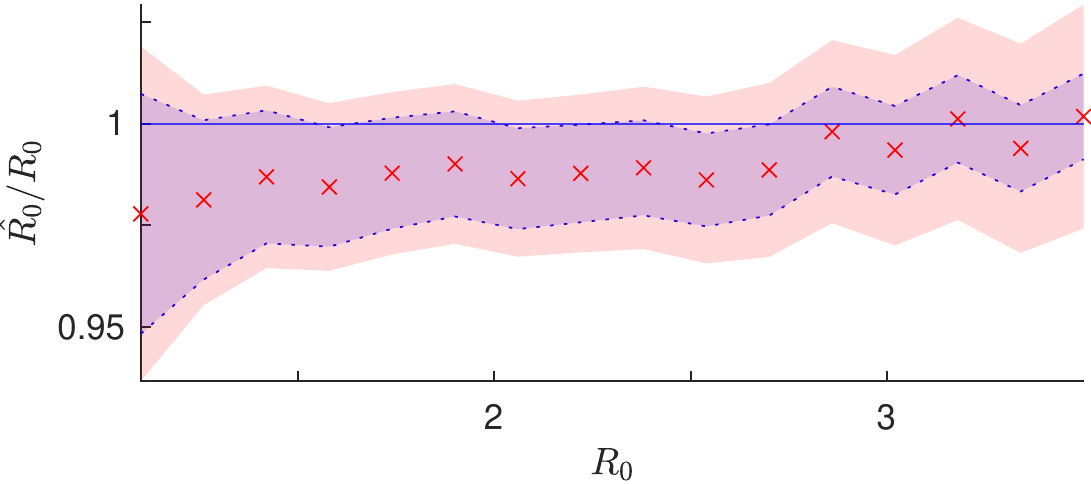}
  \includegraphics[trim = 4.5cm 12.24cm 5.4cm 12.55cm,clip =
  true]{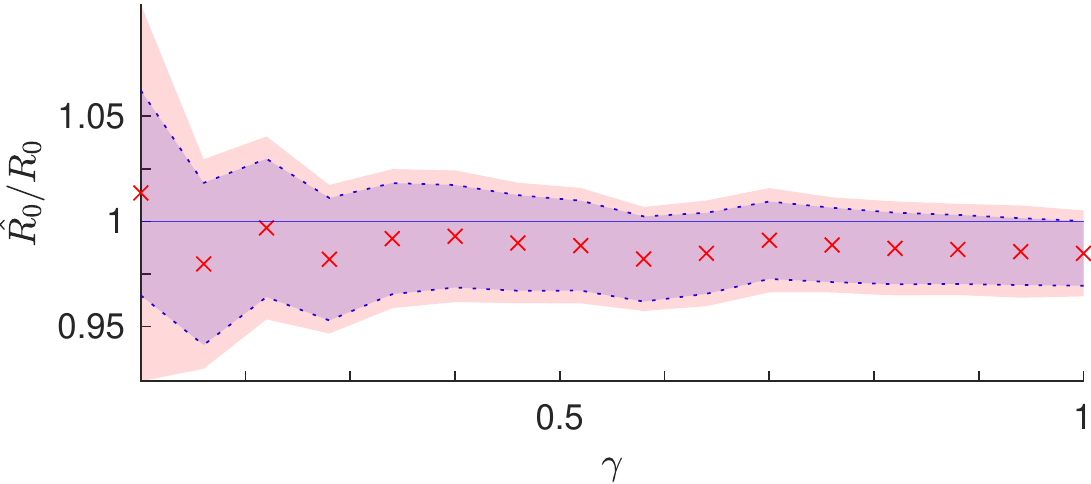}
  \caption{Marginal posterior uncertainty ($\pm 2$ SD) for the
    SIS-model and a range of parameters. \textit{Top:} with
    $\gamma \equiv 1$ fixed, \textit{bottom:} with $R_0 \equiv 1.5$
    fixed. \textit{Crosses:} MMSE-estimators $\hat{R}_0$ (i.e.,
    posterior means), \textit{dotted}: estimated uncertainty according
    to \eqref{eq:VarR0}, \textit{red}: posterior width ($\pm 2$ SD).}
  \label{fig:SIS_marginalR0}
\end{figure}

\subsection{Limits of convergence}
\label{subsec:limits_convergence}

The SIS-model investigated previously was dependent on two parameters
only and hence the singularity detected in \S\ref{sec:FilteredState}
is not likely to be limiting any convergence. While various model
modifications naturally lead to additional independent parameters,
e.g., an extra transition $S \to I$ modeling external infectious
events, the most immediate modification is to simply consider the
population size $\Sigma$ uncertain. For instance, this is a possible
setup for inference relying on sewage water analysis, where the data
is binary according to whether the infectious substance is above or
below some known threshold value $c$, but the sewage uptake area is
populated with an unknown number of individuals.

We remark that this is a considerably challenging task, and although
we are able to demonstrate the sharpness of our negative results from
\S\ref{sec:FilteredState} in this setting, the fact that this problem
can at all be approached is quite remarkable.

As ground truth we use the same parameters as in the previous section,
but with $\gamma = 1/10$ (time$^{-1}$) and $R_0 = 1.5$ fixed, and we
need to sample more, $\Ndata = 1000$ points equispaced with $h =
1$. The data is then taken to be the filtered sequence
$Y_i = \one_{X_i \ge c}$ with
$c = 0.9 \times I_\infty = 0.9 \times \Sigma(1-R_0^{-1})$.

The measurement map is strongly nonlinear and so the Kalman filter
needs to be extended in some way. We took an immediate approach by
simply discretizing the state variable $I$ into $M = 100$ Gaussian
particles, distributed according to the percentiles of the stationary
measure for a given proposed set of parameters. Each particle is
evolved in time $[0,h]$ using steps of size $\Delta t$ according to
the Kalman filter \eqref{eq:Kalman}, after which the prior
distribution is formed by aggregating the probability mass in the
vicinity of each particle. This yields the likelihood for a single
data point after which the posterior distribution for the state is
obtained by setting selected particles' mass to zero (according to the
data point) and rescaling appropriately.

From Theorem~\ref{th:nonidentifiability} we have the singular curve
$\sqrt{\gamma_{\text{\OU}}}(c-\alpha_{\text{\OU}}) = \mbox{constant}$,
which gets transferred via the map \eqref{eq:map} into a surface in
$(R_0,\Sigma,\gamma)$-space. After arbitrarily fixing $\gamma$ we thus
obtain a curve in the $(R_0,\Sigma)$-plane. Since the SIS-to-\OU\ map
\eqref{eq:SIS2OU} is an approximation and, moreover, the likelihood is
approximated via a Kalman filter, this analytical curve can be
expected to be a perturbation of the observed numerical posterior
level curves. As shown in Fig.~\ref{fig:SIS_filtered_conditionalR0}
the match is quite remarkable.

\begin{figure}[t]
  \centering
  \includegraphics{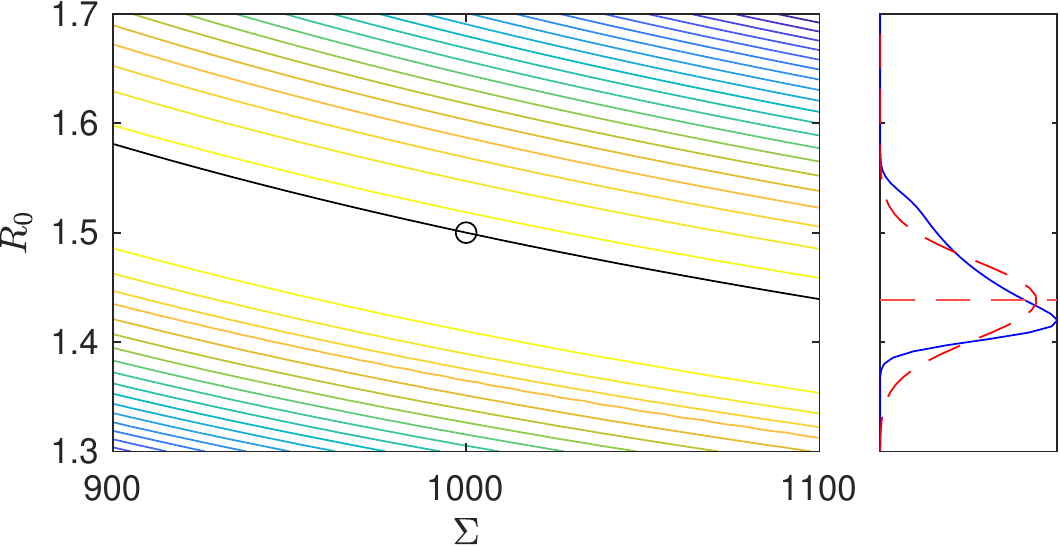}
  \caption{\textit{Left:} (log-)posterior for the SIS-model under
    filtered data and conditioned on the true value $\gamma = 1/10$
    (time$^{-1}$). Also indicated is the singular curve as predicted
    by theory passing through the true parameter generating the data
    (\textit{circle}). \textit{Right:} marginal distribution for $R_0$
    together with a normal fit (\textit{dashed}).}
  \label{fig:SIS_filtered_conditionalR0}
\end{figure}

\subsection{Network epidemics}
\label{subsec:network_epidemics}

Epidemic models on networks can give rise to phenomena not observed in
single-node systems, e.g., the \emph{rescue effect}
\cite{keeling2011modeling}, where the infection is ``rescued'' from
extinction through the network structure. Here we consider both the
SIS- and the SIS$_{\text{E}}$-model, respectively,
cf.~\eqref{eq:SIS}--\eqref{eq:SIS_E}. We assume these models at each
node in a network implicitly defined by pre-recorded movements of
individuals between the nodes. As a concrete example this would be an
appropriate model for estimating disease parameters in a monitoring
program for bovine animals using cheap, but low-informative, tests
collected on a weekly basis for a subsample of \emph{sentinel} nodes.

The nodal model is replicated across \numprint{1600} nodes, populated
with \numprint{196168} individuals, and the nodes are connected using
\numprint{466692} prescribed movements of individuals over four years,
see Figure~\ref{fig:network}. The system is not well stirred on the
aggregate level, but events occur frequently; the average \# of events
per sample node and day $= 0.20\, [0.19, 0.21]$ with 50\% credible
interval (CrI). This particular network was constructed by anonymizing
a set of recorded cattle movements and can be accessed through the
publicly available R-package SimInf~\cite{widgren2019siminf,
  engblom2020bayesian}. We extract model measurements from the same
100 randomly pre-selected sentinel nodes every 7th day for a total of
4 years. Each measurement is the outcome of a binary test: if the
prevalence ($P = \Sigma^{-1} I$) in the node is above a threshold
value ($c =$ 30\% or 4\% for the SIS- or the SIS$_{\text{E}}$-model,
respectively), where all nodes are seeded at 10\% or 2\% initially. In
Figure~\ref{fig:filtered}, we illustrate the sampling output in time
for the SIS$_{\text{E}}$-model.

\begin{SCfigure}[40][tbp]
  \includegraphics[]{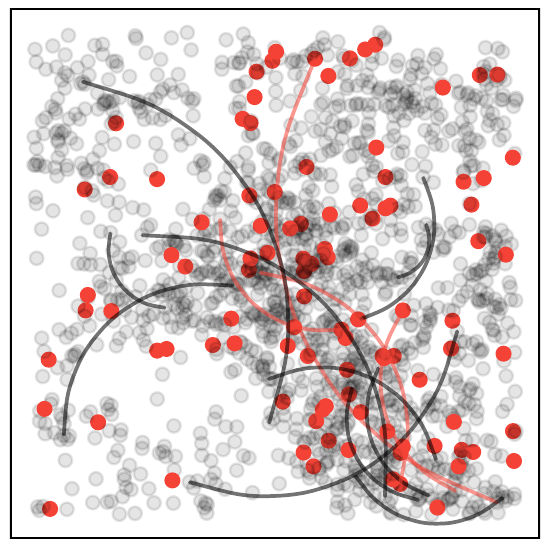}
  \caption{Illustration of the transport network; the red points are
    the sentinel nodes, and the grey points are the latent
    ones. Red/black lines are transport events into a sentinel- or
    latent node, respectively.}
  \label{fig:network}
\end{SCfigure}

\begin{figure}[tbp]
  \centering
  \includegraphics[]{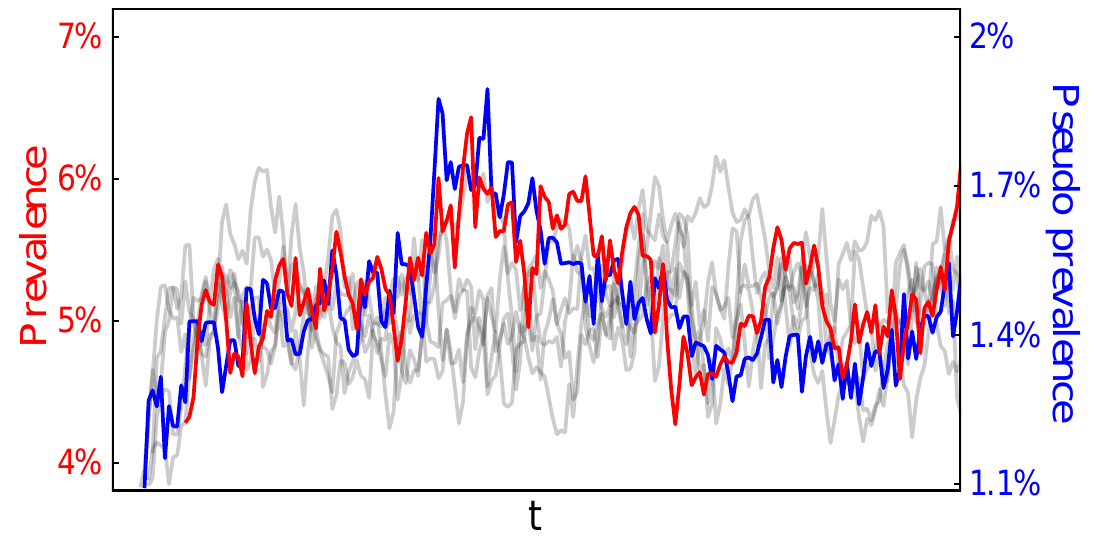}
  \caption{The population-weighted average prevalence (red) is
    unobservable, but the pseudo prevalence (blue) is obtained from
    weighting together multiple binary measurements. The least squares
    \OU-fit for the pseudo prevalence is used as summary statistics (a
    few samples in grey are shown).}
  \label{fig:filtered}
\end{figure}

A challenging aspect of many data-driven inference problems, e.g.,
including network dynamics is that the likelihood function is
intractable and must be estimated through repeated
simulations. Bayesian inference in this setting is termed
Likelihood-free inference (LFI), or Approximate Bayesian Computations
(ABC); see~\cite{marin2012approximate, sisson2018handbook} for
reviews. In this example, we consider the Sequential Monte Carlo (SMC)
adaptation of ABC (SMC-ABC) implemented in SimInf and described in
\cite{toni2009approximate}.

Our SMC-ABC implementation determines proposal rejections per
generation $n$ using the normalized Euclidean kernel
$K_\varepsilon(x,y) = \sqrt{\sum_i ((x_i-y_i)/x_i)^2} < \varepsilon_n$
for statistics of the simulation proposal $y$ and observation $x$ and
with a series of decreasing tolerances $\varepsilon_n$. The statistics
are computed as follows. Each simulation generates a time series of
\emph{pseudo prevalences}, i.e., a population-weighted sum of positive
samples. We interpret the time series as an \OU-process and select the
least square estimates \eqref{eq:lsq1}--\eqref{eq:lsq2} as
\emph{indirect} summary statistics~\cite{drovandi2011approximate}. A
word in favor of this particular choice of statistics for other ABC
implementations, e.g., synthetic likelihoods
\cite{wood2010statistical}, is that least square estimates are
asymptotically normally distributed under broad
assumptions~\cite{NEWEY19942111}. Notably, we get away with using
statistics with one more or equal dimension as the parameter set,
suggesting that this characterization is indeed very fitting.



For the inference we use a single initial simulation with the true
parameters $(\beta,\gamma,R_0)_{\text{SIS}} = (0.16, 0.1, = 1.6)$ and
$(\beta,\gamma,\rho,R_0)_{\text{SIS}_{\text{E}}} = (0.054, 0.1, 0.44,
= 1.108)$, respectively, \updated{and we infer all parameters
  simultaneously.}  For priors, \updated{since we have no likelihood
  and thus cannot easily produce a strictly non-informative prior, we
  take uniform distributions over quite large intervals in parameter
  space:} $\beta$ and $\gamma \sim \mathcal{U}(0,1)$ in both cases and
$\rho \sim \mathcal{U}(0.4,0.5)$.  We use the decreasing ABC
tolerances $\varepsilon_n = 100 \exp(-0.25 (n-1))$, for
$n = 1\ldots 15$ and \numprint{1000} SMC particles.

  We found that $\rho$ requires a tighter prior than the others
  for the computations to complete in a reasonable time: the
  SIS$_{\text{E}}$-model is considerably more challenging but is also
  more realistic, particularly so in our setting on a network where
  rather large prevalences are required for the SIS-model to not
  simply die out.

The results for the $R_0$-marginals are displayed in
Fig.~\ref{fig:SIS_network}, where we also investigate the
concentration effect of data through the relative change in quartile
coefficient of dispersion (QCD); a small concentration factor
indicates an accurately identifiable parameter. Although the
SIS$_{\text{E}}$-model is clearly more challenging, $R_0$ is still
well reconstructed for both models.

Binary data implies identifiability when given multiple observations
at the same time, e.g., over a small collection of nodes in a network
rather than on a single node, as indicated in a qualitative sense by
Theorem~\ref{th:doublefilter}. Additionally, $R_0$ is identifiable
with quite high accuracy even when the dependent parameters
covaries. The example is prototypical of using synthetic data to
evaluate the feasibility of an intended setup. Since the posteriors
are robust with respect to capturing the synthetic truth, we have good
reasons to also have some faith in the design when approaching real
data.


\begin{figure}[t]
  \centering
  \includegraphics[]{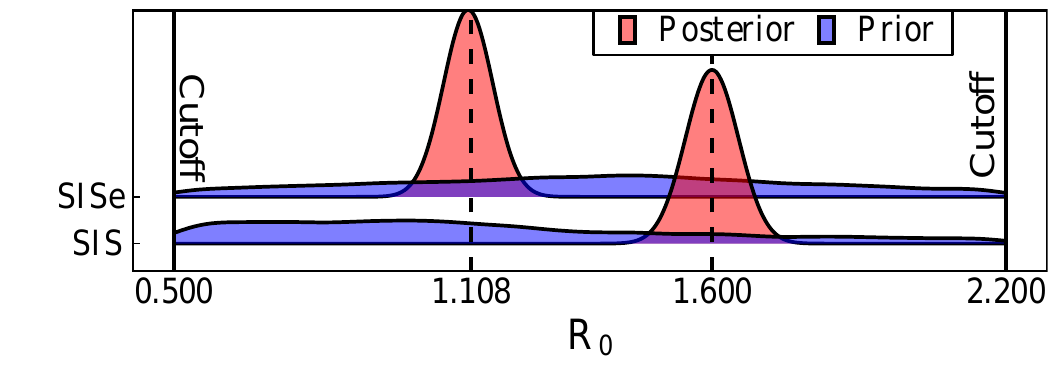}
  \begin{tabular}{l r r r r} \hline
    & $\beta$ & $\gamma$ & $\rho$  & $R_0$ \\ \hline
    SIS$_\text{E}$& 0.61    & 0.59      & 0.70   & 0.0086 \\
    SIS           & 0.089   & 0.093    & -     & 0.016 \\ \hline
  \end{tabular}
  \caption{\textit{Top:} posterior and prior distributions for $R_0$
    with the true values indicated. \textit{Bottom:} recorded QCD
    concentration factors for all the parameters (see text for
    details).}
  \label{fig:SIS_network}
\end{figure}


\section{Conclusions}
\label{sec:conclusions}

\updated{Throughout this work we have employed the Ornstein-Uhlenbeck
  process as a meta-model of more involved epidemiological models. We
  indicated in \S\ref{sec:UnfilteredState} a convergence analysis of
  the Bayesian posterior under direct process observations. Since this
  is an unrealistic setup in most epidemiological applications, one
  can think of these results as \emph{best possible}.}

\updated{We next took the opposite standpoint and considered data to
  be severely filtered such that, literally, each data point
  contributed only a single bit of information. For instance, this
  could be a model of pooled data obtained through environmental
  sampling and subsequent analysis. To obtain a closed framework we
  added a fairly general stationary $p$-step Markov assumption and
  worked out conditions on the data to obtain a non-singular inverse
  problem.}

\updated{With increasing compute power and improving possibilities for
  gathering data, fully Bayesian first-principle epidemiological
  models can be realized. As a minimum standard, we propose, such
  methods should be preceded by a proof of self-consistency: data
  generated from the model itself and chosen ``nearby'' the actual
  data should allow for accurate parameter identifiability. Under this
  basic standard, Bayesian epidemiological modeling with both
  short-term prediction and generation of forecasting scenarios can be
  included as an integrated part of the public health's methodological
  arsenal.}


\if\ams1

\section*{Acknowledgment}

This work was financially supported by the ENS Paris (S.~Bronstein)
and in part by the Swedish Research Council Formas (S.~Engblom) and
the Swedish Innovation Agency Vinnova (S.~Engblom, R.~Marin). Several
helpful comments and kind suggestions on earlier versions of the
manuscript were offered by friends and colleagues. These helped to
improve the presentation considerably.

\newcommand{\doi}[1]{\href{http://dx.doi.org/#1}{doi:#1}}
\bibliographystyle{abbrvnat}
\bibliography{MainBib}

\else

\section*{Acknowledgment}

This work was financially supported by the ENS Paris (S.~Bronstein)
and in part by the Swedish Research Council Formas (S.~Engblom) and
the Swedish Innovation Agency Vinnova (S.~Engblom, R.~Marin). Several
helpful comments and kind suggestions on earlier versions of the
manuscript were offered by friends and colleagues. These helped to
improve the presentation considerably.

\section*{Conflict of interest}

The authors declare there is no conflict of interest.

\bibliography{MainBib}

\fi

\end{document}